\documentclass[12pt]{amsart}
\usepackage{lmodern,graphicx}
\usepackage{latexsym}
\usepackage{amsfonts}
\usepackage{pgfplots}
\usepackage{tikz}
\usetikzlibrary{arrows,shapes,positioning}
\usetikzlibrary{decorations.markings}
\tikzstyle arrowstyle=[scale=1]
\tikzstyle directed=[postaction={decorate,decoration={markings,
    mark=at position .65 with {\arrow[arrowstyle]{stealth}}}}]
\tikzstyle reverse directed=[postaction={decorate,decoration={markings,
    mark=at position .45 with {\arrowreversed[arrowstyle]{stealth};}}}]
\usetikzlibrary{backgrounds}
\usepackage[german,english]{babel}

\usepackage{amsmath,amsthm,amssymb}

\newtheorem{thm}{Theorem}[section]
\newtheorem{lemma}[thm]{Lemma}
\newtheorem{cor}[thm]{Corollary}

\newtheorem{prop}[thm]{Proposition}

\begin{document}

\bigskip
\begin{center}
\Large
On transitivity and (non)amenability of $\operatorname{Aut} F_n$ actions on group presentations
\normalsize

\bigskip

Aglaia Myropolska and Tatiana Nagnibeda\footnote{This research was partly funded by the Swiss National Science Foundation, grant 200021\_144323.}
\bigskip
\end{center}

\begin{center}
{\it For Pierre de la Harpe on the occasion of his 70-th birthday, with great respect}
\end{center}
\bigskip

\begin{center}
\textbf{Abstract} \end{center} 
For a finitely generated group $G$ the Nielsen graph $N_n(G)$, $n\geq \operatorname{rank}(G)$, describes the action of the group $\operatorname{Aut}F_n$ of automorphisms of the free group $F_n$ on generating $n$-tuples of G by elementary Nielsen moves. The question of (non)amenability of Nielsen graphs is of particular interest in relation with the open question about Property $(T)$ for $\operatorname{Aut}F_n$, $n\geq 4$. We prove nonamenability of Nielsen graphs $N_n(G)$ for all $n\ge \max\{2,\operatorname{rank}(G)\}$ when $G$ is indicable, and for $n$ big enough when $G$ is elementary amenable. We give an explicit description of $N_d(G)$ for relatively free (in some variety) groups of rank $d$ and discuss their connectedness and nonamenability. Examples considered include free polynilpotent groups and free Burnside groups.

\bigskip

\section{Introduction}
Let $G$ be a finitely generated group. The following transformations of the set $G^n, n\geq 1$, were introduced by J.\ Nielsen in \cite{Ni18} and are known as \emph{elementary Nielsen moves}:
\begin{align*}
R_{ij}^{\pm}(g_1,\dots,g_i,\dots, g_j,\dots, g_n)&=(g_1,\dots,g_{i}g_j^{\pm 1},\dots,g_j,\dots, g_n),\\
L^{\pm}_{ij}(g_1,\dots,g_i,\dots, g_j,\dots, g_n)&=(g_1,\dots,g_j^{\pm 1}g_{i},\dots,g_j,\dots, g_n),\\
 I_j(g_1,\dots,g_j,\dots,g_n)&=(g_1,\dots,g_j^{-1},\dots,g_n),\\
\end{align*}
where $1\leq i,j \leq n$, $i\neq j$.  These transformations can be seen as elements of $\operatorname{Aut}F_n$; moreover, Nielsen proved that they generate $\operatorname{Aut}F_n$ (see e.g. \cite{LS}, Chap.\ I, Prop.\ 4.1). 
Hence elementary Nielsen moves transform generating sets of $G$ into generating sets. Two generating sets $U$ and $V$ are called \emph{Nielsen equivalent} ($U\sim V$) if one is obtained from the other by a finite chain of elementary Nielsen moves. 

\emph{The rank $\operatorname{rank}(G)$} of a group $G$ is the minimal number of generators of $G$. 
\medskip
We define the \emph{Nielsen graph} (also called the extended product replacement graph)  $N_n(G)$, $n\geq \operatorname{rank}(G)$, as follows: \begin{itemize} \item[-] the set of vertices consists of generating $n$-tuples in $G$, \begin{equation*}V(N_n(G))=\{(g_1,\dots,g_n)\in G^n\mid \langle g_1,\dots,g_n \rangle = G\};\end{equation*}
\item[-]  for every generating $n$-tuple $(g_1, \dots, g_n)$ and for every $(i,j)$, $1\leq i,j\leq n$, $i\neq j$, there is an edge corresponding to each of the elementary Nielsen moves $R_{ij}^{\pm}(g_1,\dots,g_n)$, $L_{ij}^{\pm}(g_1,\dots,g_n)$, $I_j(g_1,\dots,g_n)$.
\end{itemize}

The set $G^n$ of $n$-tuples in $G$ can be identified with the set of homomorphisms from the free group $F_n$ to $G$, and the set of generating $n$-tuples is then identified with the set of epimorphisms $\operatorname{Epi}(F_n,G)$. Hence there are natural actions of the automorphism group $\operatorname{Aut}F_n$ on both $G^n$ and $\operatorname{Epi}(F_n,G)$, by precomposition. 
Observe that the graph $N_n(G)$ is connected if and only if the action of $\operatorname{Aut}F_n$ on $\operatorname{Epi}(F_n,G)$ is transitive.  

Recall that, for a given group $G$ generated by a finite set $S$, and a set $M$ with a transitive action of $G$ on $M$, one can define the Schreier graph $Sch(G,M,S)$: the vertex set of the graph is $M$, and there is an edge connecting $m_1$ to $m_2$ for each $s\in S\cup S^{-1}$ that maps $m_1$ to $m_2$. Hence, if the action of $\operatorname{Aut}F_n$ on $\operatorname{Epi}(F_n,G)$ is transitive, then $N_n(G)$ is precisely the Schreier graph of $\operatorname{Aut}F_n$ acting on $\operatorname{Epi}(F_n,G)$ with respect to the elementary Nielsen moves. The set $\operatorname{Epi}(F_n,G)$ can also be understood as the set of left cosets of the subgroup $St_{\operatorname{Aut}F_n}(g_1,\dots,g_n)$ for some (any) generating $n$-tuple $(g_1,\dots,g_n) \in G^n$, and $N_n(G)$ is thus the Schreier graph with respect to this subgroup in $\operatorname{Aut}F_n$.
More generally, if the action is not transitive, every connected component of $N_n(G)$ is the Schreier graph of 
$\operatorname{Aut}F_n$ with respect to the corresponding subgroup $St_{\operatorname{Aut}F_n}(g_1,\dots,g_n)$, where the generating $n$-tuple $(g_1,\dots,g_n)$ belongs to the considered connected component.
As any Schreier graph, $N_n(G)$ comes with an orientation and a labeling of edges by elements of the generating set. The set of elementary Nielsen moves being symmetric, orientation can be disregarded in this case.

The question of (non)amenability of infinite Nielsen graphs is of particular interest in relation with the open problem about Property $(T)$ for $\operatorname{Aut}F_n$, $n\geq 4$ \cite{LubP} (the answer is negative for $n\leq 3$, see \cite{GrLu}). Namely, if a group $G$ has Property $(T)$ then $G$ does not admit any faithful amenable transitive action on an infinite countable set $X$, in other words, every infinite Schreier graph of $G$ is nonamenable. This follows from the well-known amenability criterion in terms of existence of almost invariant vectors for the action of $G$ on $l^{2}(X)$ (see Chapter G in \cite{BHVa} for a proof in the case $X=G$).

Connectedness of Nielsen graphs has been studied in different contexts, for finite and infinite groups, see \cite{Ev06,Lub,Myro,Pak} and references therein.
Even such basic questions about structure and geometry of Nielsen graphs as the number of connected components of $N_n(G)$, whether these connected components are all isomorphic, whether they are all infinite,  and so on, remain widely open. 

A locally finite connected graph $X$ of uniformly bounded degree is \emph{amenable} if either $X$ is finite or 
 $$h(X):=\inf_{S\subset V(X)}\frac{|\partial_X(S)|}{|S|}=0,$$ where the infimum is taken over all finite nonempty subsets $S$ of the set of vertices $V(X)$ and $\partial_X(S)$ is the set of all edges connecting $S$ to its complement. The number $h(X)\geq 0$ is called the {\it isoperimetric constant} (or the Cheeger constant) of $X$. A graph with several connected components is amenable if at least one of the connected components is amenable.

We will also use Kesten characterization of amenable graphs (see e.g. \ \cite[10.3]{Woes} for the extension of Kesten's criterion of amenability to all connected regular graphs). A connected $m$-regular  graph $X$ is \emph{amenable} if and only if $\rho(X)=1$, where $\rho(X)=1/m \limsup_{k\rightarrow\infty}a_k^{1/k}\leq 1$  is the \emph{spectral radius} of $X$, with  $a_k(x)$ denoting the number of closed paths of length $k$ in $X$, based at some (any) vertex of $X$. 

In this paper, we study connectedness and nonamenability of Nielsen graphs for certain families of groups. In Section \ref{Nonamenability of Nielsen graphs of indicable groups} we discuss in detail the structure of Nielsen graphs $N_n(\mathbb Z)$, $n\geq 1$, which allows us to deduce nonamenability of all Nielsen graphs $N_n(G)$, $n\geq \max\{2, \operatorname{rank}(G)\}$, for finitely generated groups $G$ that admit an epimorphism onto $\mathbb Z$ (such groups are called {\it indicable}).

\begin{thm}
Let G be a finitely generated indicable group. Then all Nielsen graphs $N_n(G)$, $n\geq \max\{2,\operatorname{rank}(G)\}$, are nonamenable.
\label{indicable.amen}
\end{thm}

 In Section \ref{elementary amenable groups} we discuss nonamenability of Nielsen graphs for infinite finitely generated elementary amenable groups. In particular we describe in detail the structure of all Nielsen graphs of the infinite dihedral group. We also show:

\begin{thm}
Let $G$ be an infinite finitely generated elementary amenable group. Then $G$ admits an epimorphism onto a group $H$ that contains a normal subgroup isomorphic to $\mathbb{Z}^d$, $d\geq 1$, of finite index $i\geq 1$. All Nielsen graphs $N_n(G)$ are nonamenable for $n\geq \operatorname{rank}(G)+\log_2 i +1$.
\label{elemamen}
\end{thm}

Related results in this direction appear also in a recent preprint \cite{Maly} by Malyshev.

In Section \ref{Nielsen graphs of relatively-free groups} we consider Nielsen graphs of relatively free groups.
 A group is called \textit{relatively free} if it is free in a variety of groups (see Section \ref{Nielsen graphs of relatively-free groups} for a more detailed definition). For a relatively free group $G$ of rank $d$ we describe explicitly the Nielsen graph $N_d(G)$ (Theorem \ref{rel.free}). 
In particular, we show that every connected component of the Nielsen graph is isomorphic to the Cayley graph of the subgroup $T(G)\leq \operatorname{Aut} G$ of {\it tame automorphisms} of $G$. (An automorphism of a relatively free group $G$ of rank $d$ is \emph{tame} if it lies in the image of the natural homomorphism $\operatorname{Aut} F_d\rightarrow \operatorname{Aut} G$ -- see Section \ref{Nielsen graphs of relatively-free groups} for the precise definition of tameness). This implies in particular that all connected components of the Nielsen graph $N_d(G)$ are isomorphic. Their number is equal to the index of the subgroup $T(G)$ in $\operatorname{Aut} G$. We deduce the following criteria.
\begin{cor}
Let $G$ be a relatively free group of rank $d$. Then
\begin{enumerate}
\item The Nielsen graph $N_d(G)$ is connected if and only if all automorphisms of $G$ are tame;
\item $N_d(G)$ is nonamenable if and only if the group $T(G)$ of tame automorphisms of $G$ is nonamenable.
\end{enumerate}
\label{rel.free.amen}
\end{cor} 

We then use these criteria to examine Nielsen graphs of various classes of relatively free groups, in Section \ref{Examples}. We first consider free polynilpotent groups. All such groups  are indicable, so their Nielsen graphs are nonamenable by Theorem \ref{indicable.amen}.
The question about connectedness of Nielsen graphs is more complicated in this class and we examine it case by case. 

We then turn our attention to free Burnside groups. Recall that the free Burnside group $B(d,m)$ of rank $d$ and exponent $m$ is the group on $d$ generators satisfying the law $x^m=1$. These groups are torsion and thus  cannot be indicable. By a famous result of Novikov and Adyan \cite{NA68} we know that for any $d\geq 2$ and $m$ odd and large enough these groups are infinite. Adyan further showed \cite{Ad82} that $B(d,m)$ are nonamenable for any $d\geq 2$ and odd $m\geq 665$. Hence, our Theorems \ref{indicable.amen} and \ref{elemamen} are not applicable in this case. Nonamenability of $N_n(B(d,m))$ for all $n\geq d\geq 3$ and $m$ odd and large enough is proven by Malyshev \cite{Maly} using uniform nonamenability of $B(d,m)$.
 Using the work of Coulon \cite{Coul} on automorphisms of free Burnside groups, as well as some results of Moriah and Shpilrain \cite{MoSh} we deduce from Corollary \ref{rel.free.amen}:

\begin{cor}
Let $B(d,m)$ denote the free Burnside group on $d$ generators of exponent $m$. If $d\geq 2$ and $m> 2^d$ then the Nielsen graph $N_d(B(d,m))$ is not connected. For $d\geq 3$ and $m$ odd and large enough all connected components of $N_d(B(d,m))$ are isomorphic and nonamenable.
\label{nonamen.Burn.}
\end{cor}

\bigskip
The authors would like to thank Pierre de la Harpe and Rostislav Grigorchuk for valuable remarks on the first version of the paper, Christian Hagendorf for the Mathematica implementation of the graph generating code, Anton Malyshev for pointing out a missing case in Theorem \ref{justinfinite} in the first version of the paper, and the anonymous referee for the careful reading of the paper.

\section{Nonamenability of Nielsen graphs of finitely generated indicable groups}
\label{Nonamenability of Nielsen graphs of indicable groups}
The proofs of Theorem \ref{indicable.amen} and Theorem \ref{elemamen} are based on an analysis of Nielsen graphs $N_n(\mathbb{Z})$, $n\geq 1$ (see Proposition \ref{nonamen}). A description of the graph $N_2(\mathbb{Z})$ appears as Example $1.3$ in \cite{MaPa}.  We begin with a few lemmas about nonamenability of subgraphs and graph coverings that will be used to deduce Theorems \ref{indicable.amen} and \ref{elemamen}.


\begin{lemma}
Let $X$ be an infinite connected graph with uniformly bounded degree. Let $X^{'}$ be a subgraph of $X$ and suppose that there exists $D\geq 0$ such that for any vertex $x\in V(X)$ there exists a vertex $x^{'} \in V(X^{'})$ at distance at most $D$. If $X^{'}$ is nonamenable, then $X$ is nonamenable.
 \label{subforest}
\end{lemma}

\begin{proof} 
Let $S\subset V(X)$ be a finite subset of the vertex set of $X$. 
Denote by $B_D(S)=\{x \in V(X)\mid d_X(x,S)\leq D\}$ the $D$-neighborhood of the set $S$ in $X$. 

By assumption, for every vertex $s\in S$ there is at least one vertex $v\in B_D(S)\cap V(X^{'})$. Set 
$$N := \max _{v\in B_D(S)\cap V(X')} | S\cap B_D(v) | .$$ 
Then  $|B_D(S)\cap V(X^{'})|\geq  |S|/N$. 

If $d$ is a uniform bound on the vertex degree of $X$, then for each $v\in B_D(S)\cap V(X^{'})$ we can roughly estimate $N\leq d+d(d-1)+\dots+d\cdot (d-1)^{D-1}\leq d^{D+1}$ since there are at most $d$ vertices at distance $1$ from $v$, $d(d-1)$ vertices at distance $2$ from $v$, \dots, $d\cdot(d-1)^{D-1}$ vertices at distance $D$ from $v$. We conclude that $|B_D(S)\cap V(X^{'})|\geq |S|/d^{D+1}$.

Now we can estimate
\begin{align*} |B_{D+1}(S)| \geq& |B_{D}(S)|+|\partial_{X^{'}} (B_D(S)\cap V(X^{'}))| \\ 
\geq& |B_{D}(S)|+h(X^{'})|B_D(S)\cap V(X^{'})| \geq |S|+h(X^{'})|S|/d^{D+1} .
\end{align*}

By the same rough count as above, we have $|\partial{S}|\geq |B_1(S)\setminus S|\geq {|B_{D+1}(S)\setminus S|}/{d^{D+1}}$. Putting all the estimates together we get
$$\frac{|\partial S|}{|S|}\geq \frac{|B_{D+1}(S)\setminus S|}{d^{D+1}|S|}\geq h(X^{'})/d^{2D+2}$$ for any finite subset $S\subset V(X)$. Hence $X$ is nonamenable.


 \end{proof}

Recall, that a graph $X$ \emph{covers} a graph $X^{'}$ if there is a surjective graph morphism $\varphi \colon X\rightarrow X^{'}$ that is an isomorphisms when restricted to the star (a small open neighborhood) of any vertex of $X$. In this case the map $\varphi$ is called a \emph{covering map}.
\begin{lemma}
If a graph covers a nonamenable graph then it is itself nonamenable.
\label{Pasc}
\end{lemma}
\begin{proof} Let $\varphi\colon X\rightarrow X^{'}$ be a graph covering map. Since $\varphi$ is a covering, closed paths in $X$ are mapped onto closed paths in $X^{'}$.  We deduce therefore that  $a_k^{X^{'}}(\varphi(x))\geq a_k^{X}(x)$ for any $x\in V(X)$, where $a_k^{X}(x)$ is the number of closed paths of length $k$ starting from a point $x$ in $X$; and consequently $\rho(X)\leq \rho(X^{'})$. In particular if $\rho(X^{'})<1$ then $\rho(X)<1$. \end{proof}

\begin{lemma}
Let $\pi\colon G\rightarrow H$ be an epimorphism between finitely generated groups and $n\geq \operatorname{rank}(G)$. If $N_n(H)$ is connected then $N_n(G)$ covers $N_n(H)$.
\label{cover}
\end{lemma}
\begin{proof} Let us consider the map $$\varphi\colon N_n(G)\rightarrow N_n(H),$$ $$\varphi( (g_1,\dots,g_n))= (\pi(g_1),\dots,\pi(g_n))$$ and prove that it is a covering map. 

First, observe that $\varphi$ maps the star of a vertex $(g_1,\dots,g_n)$ of $N_n(G)$ bijectively onto the star of $\varphi((g_1,\dots,g_n))$ in $N_n(H)$ because the map $\varphi$ commutes with the action of $\operatorname{Aut}F_n$. 

Second, the map $\varphi$ is surjective. To see this we consider a generating $n$-tuple $(h_1,\dots,h_n)$ of $H$ and show that there exists a generating $n$-tuple of $G$ which is mapped by $\varphi$ onto $(h_1,\dots,h_n)$. By assumption $N_n(H)$ is connected, therefore for any $(s_1,\dots,s_n)\in V(N_n(G))$ its image $\varphi(s_1,\dots,s_n)$ is connected with  $(h_1,\dots,h_n)$ by a sequence of elementary Nielsen moves. As $\varphi$ commutes with the elementary Nielsen moves we conclude that $(h_1,\dots,h_n)$ is the image under $\varphi$ of some $n$-tuple in $G^n$ that belongs to the orbit of $(s_1,\dots,s_n)$ under automorphisms of $F_n$, thus is generating. \end{proof}

\noindent
{\bf Remark.} Observe that if we drop the condition that $N_n(H)$ is connected, in Lemma \ref{cover}, we are still able to conclude that each connected component of $N_n(G)$ covers some connected component of $N_n(H)$.

\begin{prop} The Nielsen graph $N_n(\mathbb{Z})$ is finite if $n=1$ and nonamenable if $n\geq 2$. 
In addition, $N_n(\mathbb{Z})$ is connected for $n\geq 1$.
\label{nonamen}
\end{prop}

\noindent
{\bf Remark.} For a finitely generated infinite group $G$, the graph $N_n(G)$, $n\geq \operatorname{rank}(G)$, is finite if and only if $G\cong\mathbb{Z}$ and $n=1$. 

Indeed, suppose that $\operatorname{rank}(G)\geq 2$. Take $n\geq \operatorname{rank}(G)$. If $N_n(G)$ is finite then in particular the group of automorphisms $\operatorname{Aut} G$ is finite which is equivalent to $G$ being a finite and central extension of $\mathbb{Z}$ \cite{Alpe}. Since any such group has infinite abelianization, it admits an epimorphism onto $\mathbb{Z}$. Then $N_n(G)$ covers the infinite graph $N_n(\mathbb{Z})$, which is in contradiction with our assumption $\operatorname{rank}(G)\geq 2$. Thus $G\cong\mathbb{Z}$ and $n=1$.

\begin{proof} Notice that the set of vertices $V(N_1(\mathbb{Z}))=\{1, -1\}$ and $I_1(1)=-1$, and therefore $N_1(\mathbb{Z})$ is finite and connected.

From now on suppose $n\geq 2$. The set of vertices of the Nielsen graph $N_n(\mathbb{Z})$ is $V(N_n(\mathbb{Z}))=\{(x_1,\dots,x_n)\mid \langle x_1,\dots,x_n\rangle=\mathbb{Z}\}=\{(x_1,\dots,x_n)\mid gcd(x_1,\dots,x_n)=1\}$. 
By the Euclid's algorithm $N_n(\mathbb Z)$ is connected.


To prove nonamenability of $N_n(\mathbb Z)$, $n\geq 2$, we will exhibit a rooted subforest $\Gamma$ in $N_n(\mathbb{Z})$ of vertex degree at least $3$ everywhere except in the roots of its components. This subforest spans all but $2n$ vertices of $N_n(\mathbb{Z})$. Nonamenability of $N_n(\mathbb{Z})$ will then follow from nonamenability of the subforest by Lemma $2.1$.

The subforest $\Gamma$ is described by its components: $\Gamma=\cup_{A, B} \Gamma_{A,B} $ where $A$ and $B$ are disjoint subsets of $\{1,\dots, n\}$ (including the empty set) and $|B|\leq n-2$ .



Let us first describe the component $\Gamma_{\emptyset,\emptyset}$ of $\Gamma$. The vertex set of $\Gamma_{\emptyset,\emptyset}$  is $$V(\Gamma_{\emptyset,\emptyset})=\{(x_1,\dots,x_n)\in \mathbb{Z}^n\mid \langle x_1,\dots,x_n\rangle=\mathbb{Z}  \text{ and } x_i>0, 1\leq i\leq n\}.$$ 
At every vertex $(x_1,\dots,x_n) \in V(\Gamma_{\emptyset, \emptyset})$, consider all the edges $\{e_{ij}(x_1,\dots,x_n)\}_{1\leq i,j\leq n}$ that correspond to $R_{ij}^{+}(x_1,\dots,x_n)$. Some of them will have to be deleted so that the graph $\Gamma_{\emptyset, \emptyset}$ has no cycles, loops or multiple edges.

Here is one way to define the set of edges to be deleted.
\begin{itemize}
\item[$\diamond$] if $R_{12}^{+}(x_1,\dots,x_n)=R_{ij}^{+}(x_1,\dots,x_n)$, $(i,j)\neq (1,2)$, delete $e_{ij}(x_1,\dots,x_n)$;
\item[$\diamond$] if $R_{21}^{+}(x_1,\dots,x_n)=R_{ij}^{+}(x_1,\dots,x_n)$, $(j,i)\neq (2,1)$, delete $e_{ij}(x_1,\dots,x_n)$.
\end{itemize}
Notice that $R_{12}^{+}(x_1,\dots,x_n)\neq R_{21}^{+}(x_1,\dots,x_n)$. Indeed, if they were equal, then $x_1+x_2=x_1$ and $x_2+x_1=x_2$, therefore $x_1=x_2=0$.
\begin{itemize}
\item[$\diamond$] If $R_{12}^{+}(x_1,\dots,x_n)=R^{+}_{ij}(y_1,\dots,y_n)$, $(i,j)\neq (1,2)$, delete $e_{ij}(y_1,\dots,y_n)$;
\item[$\diamond$] if $R_{21}^{+}(x_1,\dots,x_n)=R^{+}_{ij}(y_1,\dots,y_n)$, $(i,j)\neq (2,1)$, delete $e_{ij}(y_1,\dots,y_n)$.
\end{itemize}
Notice that $R_{12}^{+}(x_1,\dots,x_n)\neq R_{21}^{+}(y_1,\dots,y_n)$. Indeed, if they were equal then $x_1+x_2=y_1$ and $y_2+y_1=x_2$, therefore $x_1+y_2$=0, we obtain a contradiction with $x_i, y_i >0$.
\begin{itemize}
\item[$\diamond$] Otherwise, if there exist $(i_1, j_1)$, \dots, $(i_k, j_k)$ with $k\geq 2$ such that $R_{i_lj_l}^{+}(x_1,\dots,x_n)=R_{i_m j_m}^{+}(x_1,\dots,x_n)$ for $1\leq l,m\leq k$, $l\neq m$, and neither of indices $(i_l, j_l)$ or $(i_m, j_m)$ is equal to $(1,2)$ or $(2,1)$ then keep only the edge with the largest in the lexicographical order index and keep it in the graph. The same rule applies when there exist $(i_1, j_1)$, \dots, $(i_k, j_k)$ for $k\geq 2$ such that $R_{i_lj_l}^{+}(x_1,\dots,x_n)=R_{i_m j_m}^{+}(y_1,\dots,y_n)$ for $1\leq l,m\leq k$, $l\neq m$ and neither of indices $(i_l, j_l)$ or $(i_m, j_m)$ is equal to $(1,2)$ or $(2,1)$: only the edge with the index largest in the lexicographical order remains in the graph $\Gamma_{\emptyset, \emptyset}$.
\end{itemize}

We conclude that the graph $\Gamma_{\emptyset, \emptyset}$ with the given structure of edges does not have cycles, loops or multiple edges. Let the vertex $(1,\dots,1)\in V(\Gamma_{\emptyset,\emptyset})$ be the root of this graph. There are at least two edges, $e_{12}(1,\dots,1)$ and $e_{21}(1,\dots,1)$, coming out of $(1,\dots,1)$, therefore it is of degree at least $ 2$. Any other vertex $(x_1,\dots,x_n)$ in $\Gamma_{\emptyset,\emptyset}$ is of degree at least $3$:
\begin{itemize}
\item[-] if $x_1>x_2$ then $(x_1-x_2, x_2,\dots,x_n)$ is connected to $(x_1,\dots,x_n)$ by $e_{12}(x_1-x_2,x_2,\dots,x_n)$, moreover there are at least two edges coming out of $(x_1,\dots,x_n)$: $e_{12}(x_1,\dots,x_n)$ and $e_{21}(x_1,\dots,x_n)$.
\item[-] if $x_2>x_1$ then $(x_1, x_2-x_1,\dots,x_n)$ is connected to $(x_1,\dots,x_n)$ by $e_{21}(x_1, x_2-x_1, \dots,x_n)$, moreover there are at least two edges coming out of $(x_1,\dots,x_n)$: $e_{12}(x_1,\dots,x_n)$ and $e_{21}(x_1,\dots,x_n)$.
\item[-] if $x_1=x_2$, then since $(x_1,\dots,x_n)\neq (1,\dots,1)$ there exists $x_i \neq 0$, $1\leq i\leq n$, such that $x_i\neq x_1$. If $x_1>x_i$ then $(x_1-x_i,\dots,x_i,\dots,x_n)$ has to be connected to $(x_1,\dots,x_n)$ by $e_{1i}(x_1-x_i, \dots, x_i, \dots, x_n)$ unless $e_{1i}(x_1-x_i, \dots, x_i, \dots, x_n)$ is in $F_{\emptyset, \emptyset}$, which means that there is another edge coming in $(x_1,\dots,x_n)$. We deduce that there is at least one edge coming in $(x_1,\dots,x_n)$ when $x_1>x_i$. Assume now that $x_i>x_1$ then $(x_1,\dots,x_i-x_1,\dots,x_n)$ is connected to $(x_1,\dots,x_n)$ by $e_{i1}(x_1, \dots, x_i-x_1,\dots,x_n)$ unless $e_{i1}(x_1, \dots, x_i-x_1,\dots,x_n)$ was deleted, which means that there is another edge coming in $(x_1,\dots,x_n)$. We deduce that there is at least one edge coming in $(x_1,\dots,x_n)$ when $x_i>x_1$. Moreover, there are always at least two edges coming out of $(x_1,\dots,x_n)$: $e_{12}(x_1,\dots,x_n)$ and $e_{21}(x_1,\dots,x_n)$.
\end{itemize}

Consider any point $(x_1,\dots,x_n)\in V(\Gamma_{\emptyset,\emptyset})$. By construction of $\Gamma_{\emptyset,\emptyset}$ there exists a path from $(x_1,\dots,x_n)$: $(x_1,\dots,x_n)\rightarrow (x_1^{(1)},\dots.,x_n^{(1)})\rightarrow  (x_1^{(i)},\dots.,x_n^{(i)})\rightarrow...$ such that $x_1^{(i+1)}+\dots+x_n^{(i+1)}<x_1^{(i)}+...x_n^{(i)}$. This sequence terminates at $(1,\dots,1)$ since $x_i>0$, $1\leq i\leq n$, and we conclude that any point in $\Gamma_{\emptyset,\emptyset}$ is connected to $(1,\dots,1)$. Therefore $\Gamma_{\emptyset,\emptyset}$ is a connected graph without cycles, i.e., a tree, every vertex of which, except for the root, is of degree at least 3. 

More generally, for any $A\subseteq \{1,\dots,n\}$ we define the subgraph $\Gamma_{A,\emptyset}$ of $N_n(\mathbb{Z})$ with the set of vertices $$V(\Gamma_{A,\emptyset})=\{(x_1,\dots,x_n)\in \mathbb{Z}^n\mid \langle x_1,\dots, x_n\rangle=\mathbb{Z}, \text{ } x_i<0 \text{ if } i\in A \text{ and }$$ $$x_i>0 \text{ otherwise}\}$$ and the set of edges $E(\Gamma_{A,\emptyset})$  defined symmetrically to $E(\Gamma_{\emptyset,\emptyset})$. The same arguments show that it is a tree, every vertex of which, except for the root, is of degree at least $3$.

Next, we  ``lower the dimension" and define, for all disjoint subsets $A$ and $B$ of $ \{1,\dots,n\}$ (including the empty set), $|B|\leq n-2$, subgraphs $\Gamma_{A,B}$ of $N_n(\mathbb{Z})$ with  $$V(\Gamma_{A,B})=\{(x_1,\dots,x_n)\mid \langle x_1,\dots,x_n \rangle=\mathbb{Z}, \text{ } x_i<0 \text{ if } i\in A,$$ $$x_i=0 \text{ if } i\in B, \text{ and } x_i> 0 \text{ otherwise}\}.$$

At every vertex $(x_1,\dots,x_n) \in V(\Gamma_{A, B})$,  consider all the edges $\{e_{ij}(x_1,\dots,x_n)\}_{1\leq i,j\leq n}$ corresponding to $R_{ij}^{+}(x_1,\dots,x_n)$ and then delete a subset of them so the graph $\Gamma_{A,B}$ has no cycles, loops or multiple edges.

For $A=\emptyset$ the  set of edges to be deleted can be defined in the similar way as for $\Gamma_{\emptyset,\emptyset}$, but instead of using $R_{12}^{+}$ and $R_{21}^{+}$, we use $R_{i_1j_1}^{+}$ and $R_{j_1i_1}^{+}$ such that $i_1, j_1\notin B$. And for $A\neq \emptyset$ we define the edges of $\Gamma_{A,B}$ symmetrically to the edges of  $\Gamma_{\emptyset, B}$. 

Let the vertex $(\epsilon_1,\dots, \epsilon_n) \in V(\Gamma_{A,B})$ be the root of $\Gamma_{A,B}$ for $\epsilon_i=-1$ if $i\in A$, $\epsilon_i=0$ if $i\in B$, and $\epsilon_i=1$ otherwise. As before, the graph $\Gamma_{A,B}$ is a tree, and every vertex, except for the root, is of degree at least $3$.
This completes the description of $\Gamma$.

Observe that $$V(N_n(\mathbb{Z}))=V(\Gamma)\cup (\pm 1,0,\dots,0)\cup...\cup (0,\dots,0,\pm 1).$$ Nonamenability of $N_n(\mathbb{Z})$ follows from Lemma \ref{subforest}. \end{proof}

Figure $1$ represents a finite fragment of the (infinite) Nielsen graph $N_2(\mathbb{Z})$ constructed using Mathematica 9.
\begin{figure}[h]
  \centering
\includegraphics[scale=0.40]{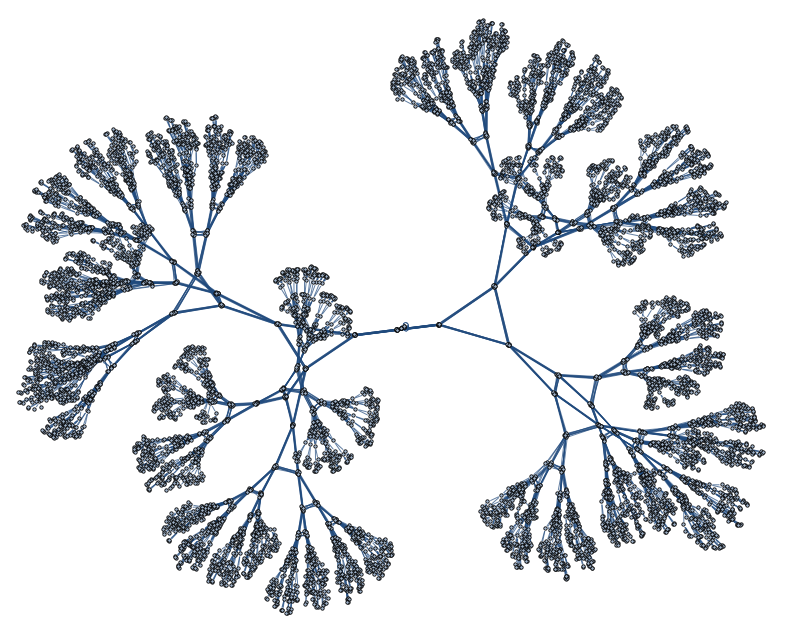}
  \caption{A finite fragment of $N_2(\mathbb{Z})$.}
\end{figure}

{\bf Remark.} The proof of Proposition gives us moreover an explicit estimate of the Cheeger constant of $N_n(\mathbb{Z})$. Let $S$ be a finite subset of vertices of $N_n(\mathbb{Z})$. If $S\subseteq \{(\pm 1,0,\dots,0) \cup... \cup (0,\dots,0,\pm 1)\}$ then $|\partial_{N_n(\mathbb{Z})}(S)|\geq (n-1)|S|$. Otherwise, let $A=S\cap \{(\pm 1,0,\dots,0) \cup... \cup (0,\dots,0,\pm 1)\}$. Then $|\partial_{N_n(\mathbb{Z})}(S)|\geq |\partial_{\Gamma}(S\setminus A)|$ and $$\frac{|\partial_{N_n(\mathbb{Z})}(S)|}{|S|}\geq \frac{|\partial_{\Gamma}(S\setminus A)|}{|S\setminus A|}\frac{|S\setminus A|}{|S|}\geq \frac{|\partial_{\Gamma}(S\setminus A)|}{|S\setminus A|}\frac{1}{2n+1}.$$ Therefore, $h(N_n(\mathbb{Z}))\geq \min\{(n-1), \frac{1}{2n+1} h(\Gamma)\}\geq \frac{1}{2n+1}$.

{\bf Remark.} Note that nonamenability of $N_n(\mathbb Z)$ for $n\geq 3$ also follows from the fact that $GL_n(\mathbb{Z})$, $n\geq 3$, has property (T). Indeed, suppose that $A$ is a finitely generated abelian group, and $n\geq \operatorname{rank}(A)$.  Then $A$ is a quotient of the free abelian group $\mathbb Z^n$. Consider the natural projection $\pi: F_n\rightarrow F_n/[F_n,F_n]$ which induces a homomorphism $\rho: \operatorname{Aut}F_n\rightarrow \operatorname{Aut}(F_n/[F_n,F_n]) = GL_n(\mathbb Z)$. Every Nielsen move defines an automorphism of $\mathbb Z^n$ which belongs to $\rho(\operatorname{Aut}F_n)$. Therefore every connected component of the Nielsen graph $N_n(A)$ is the Schreier graph \begin{equation} Sch(\rho(\operatorname{Aut}F_n), St_{\rho(\operatorname{Aut}F_n)} (a_1,\dots,a_n), \{\text{Nielsen moves}\})\end{equation}  with respect to the generating $n$-tuple $(a_1,\dots,a_n)$ which belongs to the connected component. 
Moreover observe that $\rho$ is an epimorphism \cite[3.5.1]{MaKS}. We conclude that $$(1)=Sch(GL_n(\mathbb{Z}), St_{GL_n{\mathbb{Z}}} (a_1,\dots,a_n), \{\text{Nielsen moves}\}).$$ Apply this to $A=\mathbb{Z}$ to conclude $$N_n(\mathbb{Z})=Sch(GL_n(\mathbb{Z}), St_{GL_n{\mathbb{Z}}} (x_1,\dots,x_n), \{\text{Nielsen moves}\})$$ for some generating $n$-tuple $(x_1,\dots,x_n)$ of $\mathbb{Z}$. As mentioned in the introduction, every connected infinite Schreier graph of a Property (T)-group is nonamenable. The graph $N_n(\mathbb{Z})$ is connected and infinite and therefore nonamenable. This argument does not apply to the case $n=2$.

\medskip

Proposition \ref{nonamen} allows to conclude that all Nielsen graphs $N_n(G)$ of a finitely generated indicable group $G$, 
$n\geq \max\{2, \operatorname{rank}(G)\}$, are nonamenable.

\begin{proof}[Proof of Theorem \ref{indicable.amen}] Consider an epimorphism $\pi: G\rightarrow \mathbb{Z}$. The corresponding graph morphism $N_n(G)\rightarrow N_n(\mathbb{Z})$, $n\geq \max\{2,\operatorname{rank}(G)\}$, is a covering map by Lemma \ref{cover}. We conclude by Lemma \ref{Pasc} and Proposition \ref{nonamen}. \end{proof}

\section{Nonamenability of Nielsen graphs of infinite finitely generated elementary amenable  groups}
\label{elementary amenable groups}

We begin by describing the Nielsen graphs of the infinite dihedral group $D_{\infty}=\langle r,s\mid s^2, srs=r^{-1}\rangle=\langle x,y\mid x^2, y^2\rangle$.

\begin{cor} The Nielsen graph $N_n(D_{\infty})$ is infinite connected for $n\geq 2$ and is nonamenable for $n\geq 3$. 
\label{dihedral}
\end{cor}

\begin{proof}
Recall that $D_{\infty}\cong \mathbb{Z}\rtimes \mathbb{Z}/2\mathbb{Z}$, so that $D_{\infty}/\mathbb{Z}\cong \mathbb{Z}/2\mathbb{Z}$. For any $ (x_1,\dots,x_n)\in N_n(D_{\infty})$ consider its image $(\overline{x}_1,\dots,\overline{x}_n)$ in $N_n(D_{\infty}/\mathbb{Z})$. Obviously $(\overline{x}_1,\dots,\overline{x}_n)$ is at bounded distance from $(\overline{1},\overline{0},\dots,\overline{0})$ in $N_n(D_{\infty}/\mathbb{Z})$. The same Nielsen moves which carry $(\overline{x}_1,\dots,\overline{x}_n)$ to $(\overline{1},\overline{0},\dots,\overline{0})$ will carry $(x_1,\dots,x_n)$ to $(x,y_1,\dots,y_{n-1})$ in $N_n(D_{\infty})$ for some $ x \in D_{\infty}$ and $y_i\in \mathbb{Z}$, $1\leq i\leq n-1$, such that at least one of $y_i$ is not equal to $0$ in $\mathbb{Z}$ (because $D_\infty$ is of rank 2). 

For each $x\in D_{\infty}$ and $y_1,\dots,y_{n-1} \in \mathbb{Z}$ such that $\langle x,y_1,\dots,y_{n-1}\rangle =D_{\infty}$, denote  by
$\phi_{x}$ the map from  $N_{n-1}(\mathbb{Z})$ to $N_n(D_{\infty})$ induced by the map on the vertices that sends $(y_1,\dots,y_n)$ to $(x, y_1,\dots,y_n)$.

Denote by $X'$ the subgraph $\sqcup_x \phi_x(N_{n-1}(\mathbb{Z}))$ of $N_n(D_{\infty})$. Every vertex in $N_n(D_{\infty})$ is at uniformly bounded distance from some vertex in $X'$ by the remark above. It follows from Proposition \ref{nonamen} that $X'$ is nonamenable for $n\geq 3$. By Lemma \ref{subforest} we conclude that $N_n(D_{\infty})$, $n\geq 3$, is nonamenable. 

To show connectedness of $N_n(D_{\infty})$ for $n\geq 2$ we recall that $D_\infty\cong \mathbb{Z}/2\mathbb{Z}\ast \mathbb{Z}/2\mathbb{Z}$ and evoke Grushko-Neumann's theorem \cite{Grus,Ne43} about Nielsen graphs of free products. 
\end{proof}

The picture below represents a finite fragment of the (infinite) Nielsen graph $N_2(D_{\infty})$ constructed using Mathematica 9.
\begin{figure}[h]
  \centering
\includegraphics[scale=0.8]{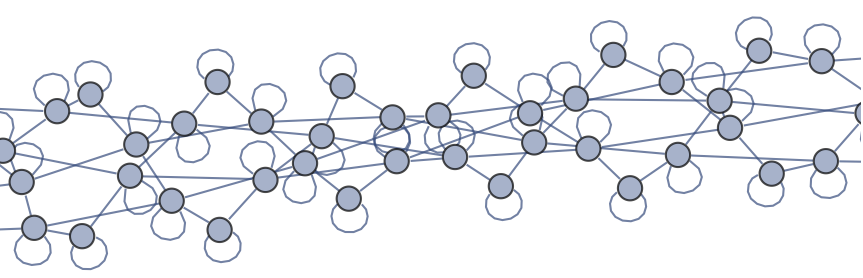}
  \caption{A finite fragment of $N_2(D_{\infty})$.}
\end{figure}
 
\begin{prop} The Nielsen graph $N_2(D_{\infty})$ is quasi-isometric to a line. In particular, it is amenable.
\end{prop}
\begin{proof}
Let $D_{\infty}=\langle a,b\mid a^2, aba=b^{-1}\rangle$. We consider the subgraph $\Gamma$ of $N_2(D_{\infty})$ whose vertex set coincides with the vertex set of $N_2(D_{\infty})$, keeping only the edges labeled by $R_{ij}(=R_{ij}^{+})$ and $I_j$, $i\neq j, 1\leq i,j\leq 2$. $\Gamma$ is a directed graph of vertex degree $8$ with loops. 
 Since $\operatorname{Aut}F_2=\langle  \{R_{ij}, I_j, i\neq j, 1\leq i,j\leq 2\}\rangle$, then $$\Gamma=Sch(\operatorname{Aut}F_2, \operatorname{Epi}(F_2, D_{\infty}), \{R_{ij}, I_j, i\neq j, 1\leq i,j\leq 2\}).$$ It follows that $\Gamma$ is quasi-isometric to $N_2(D_{\infty})$. 

\medskip

Observe that the infinite strip on Figure $3$ is a subgraph of $\Gamma$.
\begin{figure}[h]
  \centering
\includegraphics{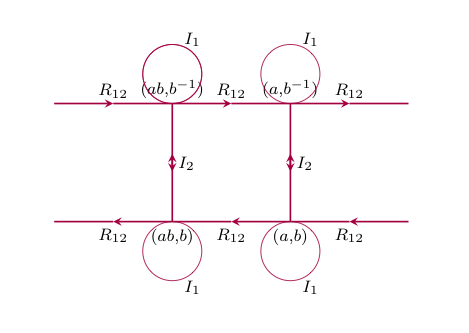}
 \caption{}
\end{figure}
Notice that all vertices on this strip are of the form $(ab^n, b^{\pm 1})$, $n\in \mathbb{Z}$, and each vertex has a loop labeled $I_1$. Indeed, $(ab)^2=abab=b^{-1}b=1$ and by induction $(ab^n)^2=ab^{n-1} aa bab^n=ab^{n-1}ab^{n-1}=(ab^{n-1})^2$.

Observe that also the infinite strip on Figure $4$ is a subgraph of $\Gamma$.
\begin{figure}[h]
  \centering
\includegraphics{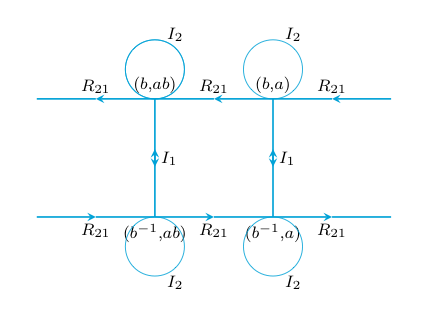}
 \caption{}
 \end{figure}
Notice that all vertices on this strip are of the form $(b^{\pm 1}, ab^n)$, $n\in \mathbb{Z}$, and each vertex has a loop labeled $I_2$. 

For any $n\in \mathbb{Z}$, $n\geq 0$, the following equalities hold:
\begin{flalign*} R_{12}R_{21}R_{12}^n(a,b)=R_{21}I_1 R_{21}^n(b,a), \\R_{21} R_{12}R_{21} I_1 R_{21}^n(b,a)=R_{12}^n (a,b), \\R_{12}I_2R_{12}^n (a,b)=R_{21}R_{12}R_{21}^n(b,a), \\R_{21}^n (b,a)= R_{12}R_{21}R_{12}I_1 R_{12}^n (a,b), \\R_{12}R_{21}(I_2R_{12}I_2)^n(a,b)=R_{21}I_1 (I_1R_{21}I_1)^n(b,a),
\\R_{21} R_{12}R_{21} I_1  (I_1R_{21}I_1)^n (b,a)=(I_2R_{12}I_2)^n (a,b), \\R_{12}I_2(I_2R_{12}I_2)^n (a,b)=R_{21}R_{12}(I_1R_{21}I_1)^n(b,a), \\(I_1R_{21}I_1)^n (b,a)= I_1R_{12}R_{21} (I_2R_{12}I_2)^{n-1} (a,b)
\end{flalign*}
We are now able to see how these two strips are connected in $\Gamma$ (see Figure $5$).

\begin{figure}[h]
  \centering
\includegraphics{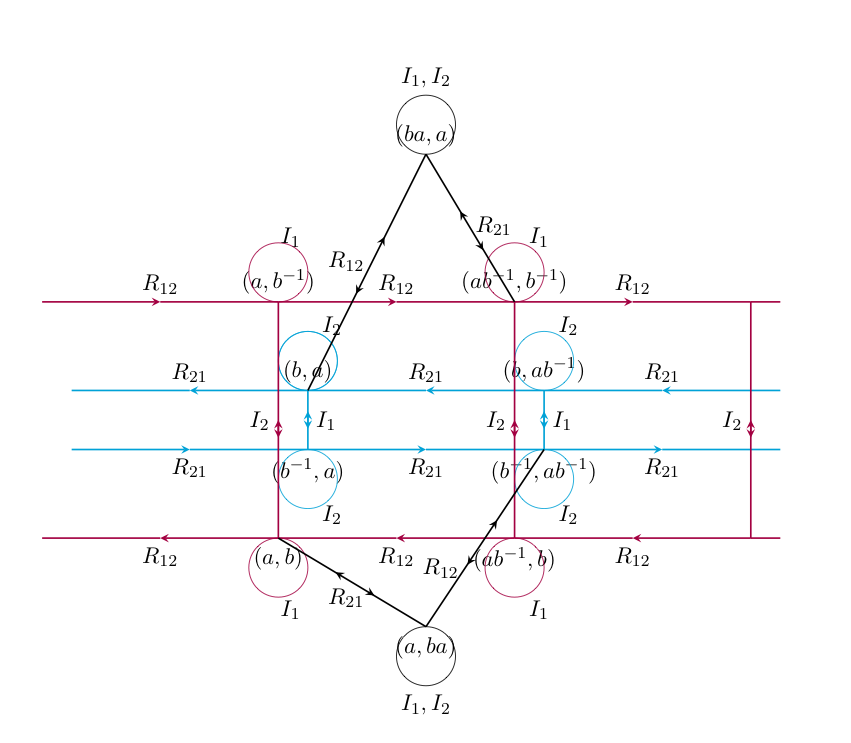}
  \caption{}
\end{figure}

\noindent Observe that for any $n\in \mathbb{Z}$, $n\geq 0$, the vertices $R_{21}R_{12}^n(a,b)$, $R_{12}R_{21}^n(b,a)$, $R_{21}(I_2R_{12}I_2)^n(a,b)$ and $R_{12}(I_1R_{21}I_1)^n(b,a)$ have loops labeled $I_1, I_2$.
The calculations above show that the graph spanned by the vertices $(ab^n, b^{\pm 1})$, $(b^{\pm 1}, ab^n)$ for $n\in \mathbb{Z}$; $R_{21}R_{12}^n(a,b)$, $R_{12}R_{21}^n(b,a)$, $R_{21}(I_2R_{12}I_2)^n(a,b)$ and $R_{12}(I_1R_{21}I_1)^n(b,a)$ for $n\in \mathbb{Z}$ and $n\geq 0$, is regular of degree $8$. Since it is a subgraph of $\Gamma$ of full degree and $\Gamma$ is connected we conclude that it coincides with $\Gamma$.

Notice that the vertices $R_{21}R_{12}^n(a,b)$, $R_{12}R_{21}^n(b,a)$, $R_{21}(I_2R_{12}I_2)^n(a,b)$ and $R_{12}(I_1R_{21}I_1)^n(b,a)$ are at distance $1$ from either $(ab^{\pm n}, b^{\pm 1})$ or $(b^{\pm 1}, ab^{\pm n})$. Thus we deduce that the graph $\Gamma$ is quasi-isometric to the line.


\end{proof}

The following Proposition will be used in the proof of Theorem \ref{elemamen}.

\begin{prop} Let $H$ be a group that contains a normal subgroup isomorphic to $\mathbb{Z}^d$, $d\geq 1$, of finite index $i>1$. Then all Nielsen graphs $N_n(H)$ are nonamenable for $n\geq \operatorname{rank}(H)+\log_2 i +1$.
\label{finiteindex}
\end{prop}

\begin{proof}
Denote by $Q$ a normal subgroup of $H$ isomorphic to $\mathbb{Z}^d$, $d\geq 1$, and denote
 by $F=H/Q$ the finite quotient of $H$. Denote also by $r=\operatorname{rank}(F)\leq \operatorname{rank}(H)$. 
For any $ (x_1,\dots,x_n)\in N_n(H)$ consider its image $(\overline{x}_1,\dots,\overline{x}_n)$ in $N_n(F)$. The Nielsen graph $N_n(F)$ is obviously finite, and it is  connected \cite[Prop. 2.2.2]{Pak} for $n\geq r+\log_2 i$. Hence
 $(\overline{x}_1,\dots,\overline{x}_n)$ is at bounded distance from $(\overline{a_1},\overline{a_2},\dots,\overline{a_r}, \overline{1},\dots,\overline{1})$ in $N_n(F)$ for $a_1,...a_r\in H$,  such that $\langle \overline{a_1},\overline{a_2},\dots,\overline{a_r}\rangle=F$. The Nielsen moves that carry $(\overline{x}_1,\dots,\overline{x}_n)$ to $(\overline{a_1},\overline{a_2},\dots,\overline{a_r}, \overline{1},\dots,\overline{1})$ in $N_n(F)$ will carry $(x_1,\dots,x_n)$ to $(a_1y_1, a_2y_2,\dots,a_ry_r, y_{r+1},\dots,y_n)$ in $N_n(H)$ for some $(y_1,\dots,y_n)\in Q^n$. 

If $y_m=1$ for all $r+1\leq m \leq n$ then $\langle a_1y_1, a_2y_2,\dots,a_ry_r\rangle=H$. A sequence of elementary Nielsen moves $R_{r+1, 1},\dots, R_{r+1, r}$ applied to $(a_1y_1, a_2y_2,\dots,a_ry_r, 1,\dots,1)$ corresponds to a path in the Cayley graph of $H$ with generators $\{a_1y_1, a_2y_2,\dots,a_ry_r\}$. Since $Q$ is of index $i$ in $H$, the ball of radius $i$ around any vertex in this Cayley graph contains at least one vertex representing an element of $Q$. We can therefore conclude that $(a_1y_1, a_2y_2,\dots,a_ry_r, 1,\dots,1)$ is within at most $i$ steps from $(a_1y_1, a_2y_2,\dots,a_ry_r, z,1,\dots,1)$ in $N_n(H)$, with $z\in Q$ and $z\neq 1$.

For each $\hat{h}=(h_1,\dots,h_r) \in H^r$ such that $\langle h_1,\dots,h_r,y_{r+1},\dots,y_{n}\rangle =H$  and $y_m\neq 1$ for some $r+1\leq m\leq n$, denote by
$\phi_{\hat{h}}$ the morphism from the graph $N_{n-r}(\langle y_{r+1},\dots,y_{n} \rangle)$ to the graph $N_n(H)$ induced by the map on the vertices that sends
$(z_{r+1},\dots,z_n)$ to $(h_1,\dots,h_r, z_{r+1},\dots,z_n)$.

As a nontrivial subgroup of $Q$, $\langle y_{r+1},\dots,y_{n} \rangle$ is isomorphic to $\mathbb{Z}^s$ for some $s\geq 1$. Notice that $n-r\geq 2$ and clearly $n-r\geq \operatorname{rank}(\langle y_{r+1},\dots,y_n\rangle)$. We use Lemma \ref{cover} and Proposition \ref{nonamen} to deduce that $N_{n-r}(\langle y_{r+1},\dots,y_{n} \rangle)$ is nonamenable. 

Denote by $X'$ the subgraph $\sqcup_{\hat{h}} \phi_{\hat{h}}(N_{n-r}(\langle y_{r+1},\dots,y_r\rangle))$ of $N_n(H)$. Every vertex in $N_n(H)$ is at uniformly bounded distance from some vertex in $X'$ by the first paragraph of the proof. Moreover $X'$ is nonamenable since each $\phi_{\hat{h}}(N_{n-r}(\langle y_{r+1},\dots,y_r\rangle))$ is nonamenable. By Lemma \ref{subforest} we conclude that $N_n(H)$ is nonamenable for $n\geq \operatorname{rank}(H)+\log_2 i +1$. 

\end{proof}

Denote by $EG$ the class of \emph{elementary amenable} groups, i.\ e., the smallest class of groups containing finite groups and abelian groups, which is closed with respect to taking subgroups, quotients, extensions and direct limits.

For each ordinal $\alpha$ define inductively a subclass $EG_{\alpha}$ of $EG$ in the following way. $EG_0$ consists of finite groups and abelian groups. If $\alpha$ is a limit ordinal then $$EG_{\alpha}=\bigcup_{\beta< \alpha} EG_{\beta}.$$ Further, $EG_{\alpha+1}$ is defined as  the class of groups which are extensions of groups from the set $EG_{\alpha}$ by groups from the same set. Each of the classes $EG_{\alpha}$ is closed with respect to taking subgroups and taking quotients \cite{Chou}. \emph{The elementary complexity} of a group $G\in EG$ is the smallest $\alpha$ such that $G\in EG_{\alpha}$.

\medskip

Recall that a group is {\it just-infinite} if it is infinite and all its non-trivial normal subgroups are of finite index. A just-infinite group $G$ is \emph{hereditary just-infinite} if it is residually finite and every subgroup $M < G$ of finite index is just-infinite.
The proof of Theorem \ref{elemamen} is based on the following trichotomy for finitely generated just-infinite groups:

\begin{thm}[\cite{G00a}]
Any finitely generated just-infinite group is either branch, or contains a normal subgroup of finite index which is isomorphic to the direct product of a finite number of copies of a group $L$, where $L$ is either simple or hereditarily just infinite.
\label{justinfinite}
\end{thm}

\emph{Branch groups} are the groups that have a faithful level transitive action on an infinite spherically homogeneous rooted tree $T_{\bar m}$ defined
 by a sequence $\{m_n\}_{n=1}^{\infty}$ of natural numbers $m_n\geq 2$ (determining the branching number of vertices of level $n$) with the property
 that the rigid stabilizer $rist_G(n)$ has finite index in $G$ for each $n\geq 1$. Here $rist_G(n)$ denotes the product $\prod_{v\in V_n}rist_G(v)$
  of rigid stabilizers $rist_G(v)$ of all vertices on the $n$-th level of the tree, where $rist_G(v)<G$ is the subgroup of elements fixing the vertex $v$ and acting trivially outside the full subtree  rooted at $v$. For more on branch groups see \cite{Gri00}. The statement of the next Proposition appeared already in \cite{Gr03} but there is no proof of it in the literature, that is why we include a proof here. 
 
\begin{prop}[R.\ I.\ Grigorchuk]
  Let $G$ be a finitely generated branch just-infinite group. Then it does not belong to the class $EG$ of elementary amenable groups.
\label{branch}
 \end{prop}

 \begin{proof}
 Suppose $G$ is branch. Let $\alpha$ be a minimal ordinal with the property $G\in EG_{\alpha}$. Then $G$ is an extension of a subgroup $N\lhd G$, $N\in EG_{\alpha-1}$
(G cannot be presented as a direct limit of subgroups of smaller elementary complexity since $G$ is finitely generated). As $G$ is just-infinite, $N$ has finite index. It turns out that $N$ is not necessarily a branch group. However it can be shown that $N$ satisfies the definition of a branch group with a single relaxation, namely, that the number of orbits of the action on the levels is uniformly bounded (instead of being equal to $1$ in the original definition). Proposition \ref{branch} will be proven by induction on $\alpha$ for any group satisfying the relaxed branch condition.
 
It is proven in Theorem $4$ \cite{Gri00}, that for each nontrivial normal subgroup $K$  of a branch group $G$ there is $n$ such that $K$ contains the commutator subgroup $(rist_G(n))'$. The same proof essentially works for groups satisfying this relaxed branch condition, one just needs to ``decompose" the tree $T_{\bar m}$ on which the group acts into finitely many invariant subtrees on each of which the action is level transitive and so the restriction of the action to each component is a branch group. As each class $EG_{\beta}$ is closed with respect to taking subgroups or quotients, the group $(rist_G(n))'$ belongs to the class $EG_{\alpha-1}$.

Consider the decomposition $rist_N(n)=\prod_{v\in V_n} rist_N(n)$. For each $v\in V_n$ the corresponding group $M_v=rist_N(v)$ satisfies the relaxed branch condition for the action on a rooted subtree $T_v$ of $T_{\bar m}$. Indeed, for each level $k$ of $T_v$ the number of orbits for the action of $rist_N(v)$ is uniformly bounded by the same constant which bounds the number of orbits of the action of $N$ on $T_{\bar m}$. Rigid stabilizer $rist_{M_v}(k)$ is a subgroup of finite index in $M_v$ as it contains the product $\prod_{u\in V_k(T_v)}rist_N(u)$ where $V_k(T_v)$ denotes the set of vertices of level $k$ in the subtree $T_v$. 

Moreover, each $M_v$ is just-infinite. Indeed, let us suppose that $P_v\lhd M_v$ is a normal subgroup. The group $Q:=\prod_{w\in V_n}P_v^{g_w}$ where elements $g_w \in G$ are chosen in such a way that $P_v^{g_w}$ is a subgroup of $rist_G(w)$, is normal not only in $N$ but also in $G$ and has infinite index. Contradiction.

Therefore $M_v$ is a finitely generated (as a quotient of the finitely generated group $rist_N(n)$) just-infinite group from the class $EG_{\alpha-1}$ that satisfies the relaxed branch condition, which gives us the final contradiction. 
 
 \end{proof}

\begin{proof}[Proof of Theorem \ref{elemamen}] Let $G$ be an infinite finitely generated elementary amenable group. As any infinite finitely generated group, it can be epimorphically mapped onto a finitely generated just-infinite group $\bar G$. The property of being elementary amenable is preserved in homomorphic images, so $\bar G$ is also elementary amenable.

We now use the classification of Theorem \ref{justinfinite}. A finitely generated just-infinite branch group cannot be elementary amenable by Proposition \ref{branch}. An infinite finitely generated simple group cannot be elementary amenable \cite{Chou}, therefore $G$ cannot contain a normal subgroup of finite index which is isomorphic to the direct product of a finite number of copies of a simple group.

 An elementary finitely generated amenable hereditary just-infinite group is isomorphic to either $\mathbb Z$ or to $D_\infty$. See Theorem 5.5 in \cite{Gr12} for a proof of this fact by Y.\ de Cornulier. Hence, any infinite finitely generated elementary amenable group is mapped onto a just-infinite group $H$ that contains a normal subgroup of finite index isomorphic either to $\mathbb Z^d$ or to $D_{\infty}^d$, $d\geq 1$. Moreover $D_{\infty}$ contains $\mathbb{Z}$ as a subgroup of index $2$, so the second case is reduced to the first. The proof is concluded via Lemma \ref{Pasc}, Lemma \ref{cover} and Proposition \ref{finiteindex}.

\end{proof}

\section{Nielsen graphs of relatively-free groups}
\label{Nielsen graphs of relatively-free groups}
A \emph{variety of groups} $\mathcal{B}$ is a class of groups that satisfy a fixed system of relations 
$$\{v=1\}_{v\in \mathcal{V}}$$ 
where $v$ runs through a set $\mathcal{V}$ of finite length freely reduced words in some alphabet $X$, called the laws of the variety. In other words, a group $G$ is in 
$\mathcal{B}$ if and only if all laws $\{v=1\}_{v\in \mathcal{V}}$ hold in $G$ when elements of $G$ are substituted for the letters.


Examples of varieties of groups include the variety of all groups defined by the empty set of laws, the variety of abelian groups defined by the commutative law $xy=yx$, nilpotent groups of a given nilpotency class, solvable groups of a given derived length and so on.  Another example is the \lq\lq Burnside\rq\rq\ variety of groups of exponent $p$ defined by the law $x^p=1$. By a theorem of Birkhoff \cite{Birk}, a class of groups is a variety if and only if it is closed under taking subgroups, homomorphic images and unrestricted direct products.

Let $\mathcal{B}$ be a variety of groups with the set of laws $\{v=1\}_{v\in \mathcal{V}}$. For an arbitrary group $G$ denote by $\mathcal{V}(G)$ the subgroup of $G$ generated by all values of words $v\in \mathcal{V}$ when elements of $G$ are substituted for letters. The subgroup $\mathcal{V}(G)$ is called the \emph{verbal subgroup} of $G$ defined by $\mathcal{V}$. It is easy to see that $G\in \mathcal{B}$ if and only if $\mathcal{V}(G)=\{1\}$. Verbal subgroups are fully invariant (i.e., invariant by all endomorphisms of the group), in particular characteristic. 



Every variety $\mathcal{B}$ of groups with the set of laws $\mathcal{V}$ contains for all $d\geq 1$ the \lq\lq relatively free group\rq\rq\ of rank $d$, which is the factor of the free group $F_d$ by its verbal subgroup $\mathcal{V}(F_d)$. Examples of relatively free groups include free groups, free abelian groups, free nilpotent groups $F_{d,c}$ of rank $d$ and nilpotency class $c$, free solvable groups $F_{d,l}$ of rank $d$ and derived length $l$, free Burnside groups $B(d,m)$ of rank $d$ and exponent $p$ and so on.

Let $F_d$ be the free group of rank $d\geq 2$ and let $\mathcal{V}$ be a verbal subgroup of $F_d$. Denote by $G$ the corresponding relatively free group $F_d/\mathcal{V}$. As $\mathcal{V}$ is characteristic, the natural mapping $\pi: F_d\rightarrow G$ induces a homomorphism \begin{equation}\label{rho} \rho: \operatorname{Aut} F_d\rightarrow \operatorname{Aut} G. \end{equation} Elements of the image of $\rho$ are called \emph{tame automorphisms} of $G$. We denote by $T(G)$ the subgroup of tame automorphisms in $\operatorname{Aut} G$. Note that  $T(G)\cong \operatorname{Aut} F_d/\operatorname{Ker}\rho$. Note also that the set $$S=\{\rho(R_{ij}^{\pm}), \rho(L_{ij}^{\pm}), \rho(I_j), 1\leq i,j \leq d, i\neq j)\}$$ of images of elementary Nielsen moves is a generating set of $T(G)$. 

\begin{lemma}
For a relatively free group $G$ there is a bijection between $\operatorname{Aut} G$ and $\operatorname{Epi}(F_d, G)$.
\label{bij}
\end{lemma}
\begin{proof}
Notice that for any $d$-generated group $G=\langle x_1,\dots,x_d\rangle$ there is a natural action of $\operatorname{Aut}G$ on $\operatorname{Epi}(F_d, G)$ by composition, and this action is free. So by fixing a generating $d$-tuple $(x_1,\dots,x_d)$ in $G$ we can map $\operatorname{Aut}G$ bijectively on the $\operatorname{Aut}G$-orbit of $(x_1,\dots,x_d)$.

We will now show that if $G$ is relatively free, then any element of $ \operatorname{Epi}(F_d,G)$ belongs to this orbit.
First, observe that since $\mathcal{V}$ is a verbal subgroup of $F_d$, $G$ has the same presentation $G=\langle g_1,\dots,g_d\mid v\in \mathcal{V}\rangle$ for any generating $d$-tuple $(g_1,\dots, g_d)$. Second, recall that two groups having the same presentation are isomorphic (see \cite{MaKS}, Theorem 1.1). From this we deduce that any generating $d$-tuple $(g_1,\dots,g_n)$ is the image of $(x_1,\dots,x_d)$ by an automorphism of $G$. \end{proof}

\bigskip
We now have the following description of the graph $N_d(G)$. 

\begin{thm} Let $G$ be a relatively free group of rank $d$. Denote by $i\in \mathbb{N}\cup \{\infty\}$ the index of the subgroup $T(G)$ of tame automorphisms in the full group of automorphisms $\operatorname{Aut} G$. Then the Nielsen graph $N_d(G)$ consists of $i$ connected components, each of them isomorphic to the Cayley graph $Cay(T(G),S)$ of $T(G)$ with respect to the set $S$ determined by the elementary Nielsen moves. \label{rel.free}
\end{thm}
\begin{proof}
Let $(g_1,\dots,g_d)$ be a generating $d$-tuple of $G$. Think about it as $\pi(x_1)=g_1, \dots, \pi(x_d)=g_d$ for a free basis $x_1,\dots,x_d$ of $F_d$ and the projection $\pi: F_d\rightarrow G$. Then for any $\sigma\in \operatorname{Aut} F_d$ the action of $\rho(\sigma)$ is given by  $\rho(\sigma)(g_k)=\pi(\sigma(x_k))$ for $1\leq k \leq d$, with $\rho$ defined by  (\ref{rho}).

We consider the action of $\operatorname{Aut} F_d$ on $\operatorname{Epi}(F_d,G)$  and prove that every connected component of the Nielsen graph $N_d(G)$ is $Cay(T(G),S)$. For this we show that $St_{\operatorname{Aut}F_d}(g_1,\dots,g_d)=\operatorname{Ker} \rho$. Assume that $\sigma \in St_{\operatorname{Aut}F_d}(g_1,\dots,g_d)$. It then defines a trivial map on generators and therefore a trivial automorphism of $G$. Hence $\sigma \in \operatorname{Ker}\rho$.
Conversely, if $\sigma \in \operatorname{Ker}\rho$ then $\rho(\sigma)(g_1,\dots,g_d)=(\pi(\sigma(x_1),\dots,\pi(\sigma(x_d))=(\pi(x_1),\dots,\pi(x_d))$ and by definition of action of $\sigma \in \operatorname{Aut}F_d$ on $\operatorname{Epi}(F_d,G)$, as explained in introduction, $\sigma \in St_{\operatorname{Aut}F_d}(g_1,\dots,g_d)$. Since the subgroup $\operatorname{Ker}\rho$ is normal in $\operatorname{Aut}F_d$, we conclude that every connected component of $N_d(G)$ is the Cayley graph $Cay(T(G),S )$.

Assume that two generating $d$-tuples $U_1$ and $U_2$ lie in different connected components of $N_d(G)$, i.e.\, $\forall \sigma \in \operatorname{Aut}F_d$ we have $U_1^{\sigma}\neq U_2$. By Lemma \ref{bij} the tuples $U_1$ and $U_2$ define automorphisms of $G$, namely, $U_1=\varphi_1(g_1,\dots,g_d)$, $U_2=\varphi_2(g_1,\dots,g_d)$ for some $\varphi_1, \varphi_2 \in \operatorname{Aut} G$. Since $U_1$ and $U_2$ are not Nielsen equivalent we have $\rho(\sigma)\varphi_1(g_1,\dots,g_d)\neq \varphi_2(g_1,\dots,g_d)$ for all $\sigma \in \operatorname{Aut}F_d$. Therefore two automorphisms  define two different connected components if and only if they lie in different right cosets of the subgroup $T(G)$ in $\operatorname{Aut} G$. We conclude that the number of connected components is equal to the index $[\operatorname{Aut} G : T(G)]$. \end{proof}

\medskip
We deduce Corollary \ref{rel.free.amen} from Theorem \ref{rel.free}. Namely, for a relatively free group $G$ of rank $d$, the Nielsen graph $N_d(G)$ is connected if and only if all automorphisms of $G$ are tame; in addition, $N_d(G)$ is nonamenable if and only if the group $T(G)$ is nonamenable.

\bigskip


Let now $H$ be  a quotient of a relatively free group $G$ of rank $d$.  Then every connected component of the Nielsen graph $N_d(H)$ is the Schreier graph $Sch(\rho(\operatorname{Aut}F_d), St_{\rho(\operatorname{Aut}F_d)} (h_1,\dots,h_d), \{\text{Nielsen moves}\})$ for some generating $d$-tuple that belongs to the connected component. (This has been observed in \cite[Prop.1.10]{LubP} for finite groups.) For infinite Nielsen graphs we get the following sufficient condition of nonamenability that replaces the criterion (2) in Corollary \ref{rel.free.amen} for quotients of relatively free groups -- recall the discussion from the introduction about the link between Property (T) of a  group and nonamenability of its infinite Schreier graphs.  

\begin{cor} Let $H$ be a finitely generated group in some variety of groups $\mathcal{B}$. For $d\geq \operatorname{rank}(H)$ denote by $G$ the relatively free group in $\mathcal{B}$ of rank $d$. If the subgroup $T(G)<\operatorname{Aut} G$ of tame automorphisms of $G$ has Property (T), then every infinite component of the Nielsen graph $N_d(H)$ is nonamenable.
\label{cor.3}
\end{cor} 

Moreover, Lemma \ref{Pasc} and Lemma \ref{cover} imply the following Corollary (generalizing Theorem \ref{indicable.amen} for $n\geq 3$).

\begin{cor} Let $K$ be a finitely generated group that admits an epimorphism onto a group $H$ belonging to some variety of groups $\mathcal{B}$. Let $d\geq \operatorname{rank}(K)$ and denote by $G$ the relatively free group of rank $d$ in $\mathcal{B}$. If $T(G)$ has Property (T), and if $N_d(H)$ is infinite and connected, then every connected component of $N_d(K)$ is nonamenable.
\end{cor}

\section{Examples}
\label{Examples}
Let us first look at the basic example of relatively free groups: free groups $F_d$ of rank $d\geq 2$. Notice that they are indicable and hence by Theorem \ref{indicable.amen} their Nielsen graphs are nonamenable.  Also Nielsen graphs $N_n(F_d)$ of free groups are connected for all $n\geq d\geq 1$. Indeed, for $d=1$ see Proposition~$2.1$ above. For $d= 2$, by Grushko-Neumann theorem \cite{Grus,Ne43}, any generating $n$-tuple of the free product of two groups $G_1$ and $G_2$ can be obtained from a set of generators, a part of which lies in $G_1$ and the rest lies in $G_2$, by a Nielsen move (see \cite[Section 4.1]{MaKS}). If we let $G_1=G_2=\mathbb{Z}$ then this implies that the Nielsen graph $N_n(F_2)$ of the free group $F_2$ is connected for $n\geq 2$. Similarly $N_n(F_d)$ is connected for $n\geq d$.

Recall that  a group $G$ is called \textit{polynilpotent} (\cite{Smir}) if it admits a finite normal series $G\geq G_{m_1}\geq G_{m_1,m_2}\geq...\geq 1$ where $G_{m_1}$ is the $m_1$-th member of its lower central series, $G_{m_1,m_2}$ is the $m_2$-th member of the lower central series of the group $G_{m_1}$ and so on. A free polynilpotent group $G=F_d/(F_d)_{m_1,\dots,m_k}$ is indicable and hence by Theorem \ref{indicable.amen} all its Nielsen graphs are nonamenable. 



We consider separately the cases of free abelian, free nilpotent, free (nilpotent of class $2$)-by-abelian, free metabelian and free centre-by-metabelian groups to describe what is known about connectedness of Nielsen graphs of free polynilpotent groups.
 
It is well known that the map $\operatorname{Aut}F_d\rightarrow GL_d(\mathbb{Z})$ is onto, so that all automorphisms of the free abelian groups are tame. 
Moreover, not only $N_d(\mathbb{Z}^d)$ but all Nielsen graphs of free abelian groups are connected. Indeed, view $\mathbb{Z}^d$ as a $\mathbb{Z}$-module. Then for any generating $n$-tuple $(v_1,\dots,v_n)$, the vectors $v_1,\dots,v_n$ are linearly dependent. Without loss of generality let $v_1,\dots,v_d$ be a linearly independent set of vectors that generates $\mathbb{Z}^d$ and deduce that $(v_1,\dots,v_n)\sim (e_1,\dots,e_d,1,\dots,1)$ where $e_1,\dots,e_d$ is the standard basis for $\mathbb{Z}^d$.

For free nilpotent groups $F_{d,c}$ of rank $d$ and nilpotency class $c$, all automorphisms are tame when $c=1$ and $c=2$. If $c=1$ then $F_{d,1}=\mathbb{Z}^d$ and if $c=2$ then $F_{2,2}$ is the Heisenberg group $\mathcal{H}_1=\langle x,y\mid [x,[x,y]],[y,[x,y]]\rangle$ (see \cite{ Myro} for connectedness of Nielsen graphs of Heisenberg groups). It has been shown however that when $c\geq 3$, the group $\operatorname{Aut}F_{d,c}$ contains non-tame automorphisms \cite{Andreadakis,Bach}. In particular it can be shown that $N_2(F_{2,3})$ contains infinitely many connected components (see \cite{Myro}). 
On the other hand, Evans proved (\cite{Evans}) that if $G$ is a nilpotent group of rank $d$ then the Nielsen graph $N_n(G)$ is connected for all $n\geq d+1$.

For free (nilpotent of class $2$)-by-abelian groups $G_d=F_d/[F_d^{'},F_d^{'},F_d^{'}]$, Gupta and Levin \cite{GupL} proved that the group $\operatorname{Aut} G_4$ contains countably many non-tame automorphisms. Papistas \cite{Pap} extended their result to $d\geq 4$ and also showed that in the case $d=2$ and $d=3$ the group $\operatorname{Aut} G_d$ is not finitely generated. Therefore $N_2(G_2)$ and $N_3(G_3)$ have infinitely many connected components by Corollary \ref{rel.free.amen}.

For free metabelian groups $M_d=F_d/[\gamma_2(F_d),\gamma_2(F_d)]$, where $\gamma_2(F_d)$ is the second derived subgroup, Bachmuth and Mochizuki \cite{BM82, BM85} proved that $M_2$ and $M_d$, $d\geq 4$, have only tame automorphisms. However Chein \cite{Chei} showed that $M_3$ has non-tame automorphisms, and moreover $\operatorname{Aut}M_3$ is not finitely generated \cite{BM82}. Corollary \ref{rel.free.amen} then implies that there are infinitely many connected components in $N_3(M_3)$. 

For free centre-by-metabelian groups $G_d=F_d/[\gamma_2(F_d), F_d]$, St\"ohr \cite{Stoh} proved that $\operatorname{Aut} G_d$ is not finitely generated for $d=2$ and $d=3$, so we can again conclude by Corollary \ref{rel.free.amen} that there are infinitely many connected components of $N_d(G_d)$ for $d=2$ or $d=3$. For $d\geq 4$, the group $\operatorname{Aut} G_d$ is generated by tame automorphisms and at most one additional automorphism \cite{Stoh}, but the question whether all automorphisms are tame remains open.


\medskip
Let us now consider free Burnside groups $B(d,m)=F_d/F_d^m$ where $F_d^m$ is the verbal subgroup of $F_d$ generated by the law $x^m=1$, $m\geq 2$, $d\geq 2$. 

In \cite{Coul}, Coulon proves the following theorem.
\begin{thm}\cite{Coul}
Let $d\geq 3$. There exists an integer $m_0$ such that for all odd $m$ larger than $m_0$, the group $\operatorname{Out}B(d,m)$ of outer automorphisms of $B(d,m)$ contains a subgroup isomorphic to $F_2$. 
\end{thm}

It follows from Coulon's proof that the free subgroup that he finds in $\operatorname{Out}B(d,m)$ is in fact a subgroup of induced tame automorphisms. Indeed, the injective homomorphism $F_2\hookrightarrow \operatorname{Out} F_d/F_d^m$ that he constructs is induced by a homomorphism $F_2\rightarrow \operatorname{Out}F_d$. In particular we can conclude that $T(B(d,m))$ is nonamenable. 

On the other hand one can show that Burnside groups also possess non-tame automorphisms. Let us first consider the case of $d=2$ and suppose $m\geq 5$. Take, for example, an integer $q$ such that $q$ and $m$ are coprime, and $1<q^2<m-1$. Let $(x_1, x_2)$ be a generating set of $B(2,m)$. The map $x_1\rightarrow x_1^q$, $x_2\rightarrow x_2^q$ can be extended to an automorphism of $B(2,m)$ which is not tame (see \cite[Remark $0.2$]{MoSh} for details). Similarly for each $d\geq 2$ and for odd $m>2^d$ there are non-tame automorphisms of $B(d,m)$. 

The two parts of Corollary \ref{rel.free.amen} now imply Corollary \ref{nonamen.Burn.}. Namely, if $d\geq 2$ and $m> 2^d$ the Nielsen graph $N_d(B(d,m))$ is not connected. And for $d\geq 3$ and $m$ odd and large enough all connected components of $N_d(B(d,m))$ are isomorphic and nonamenable.

\bigskip
Section de Math\'ematiques, Universit\'e de Gen\`eve, 2-4 rue du Li\`evre, 1211 Gen\`eve, Switzerland.  

\vspace{0.5 cm}
\emph{Aglaia.Myropolska@unige.ch; Tatiana.Smirnova-Nagnibeda@unige.ch}

\end{document}